\documentclass{amsart}

\usepackage[leqno]{amsmath}
\usepackage{latexsym,amsthm,amsfonts,amssymb,hyperref}
\theoremstyle{plain}
\newtheorem{theorem}{Theorem}[section]
\newtheorem{claim}[theorem]{Claim}
\numberwithin{equation}{theorem}
\renewcommand{\labelenumi}{(\emph{\roman{enumi}})}

\theoremstyle{definition} 
\newtheorem{remark}[theorem]{Remark}

\newcommand{\fp}{\mathfrak{p}}

\newcommand{\injdim}[2][R]{\operatorname{inj.dim}_{#1}#2}
\newcommand{\projdim}[2][R]{\operatorname{proj.dim}_{#1}#2}
\newcommand{\ZZ}{\mathbb{Z}}
\newcommand{\is}{\cong}
\newcommand{\lra}{\longrightarrow}
\newcommand{\QQ}{\mathbb{Q}}
\newcommand{\QZ}{\QQ/\ZZ}
\renewcommand{\H}[2][]{\operatorname{H}_{#1}(#2)}
\newcommand{\splf}[1][R]{\operatorname{splf}#1}
\newcommand{\Hom}[3][R]{\operatorname{Hom}_{#1}(#2,#3)}
\newcommand{\Ext}[4][R]{\operatorname{Ext}_{#1}^{#2}(#3,#4)}
\newcommand{\tp}[3][R]{\nobreak{#2\otimes_{#1}#3}}

\def\soft#1{\leavevmode\setbox0=\hbox{h}\dimen7=\ht0\advance \dimen7
  by-1ex\relax\if t#1\relax\rlap{\raise.6\dimen7
    \hbox{\kern.3ex\char'47}}#1\relax\else\if T#1\relax
  \rlap{\raise.5\dimen7\hbox{\kern1.3ex\char'47}}#1\relax \else\if
  d#1\relax\rlap{\raise.5\dimen7\hbox{\kern.9ex
      \char'47}}#1\relax\else\if D#1\relax\rlap{\raise.5\dimen7
    \hbox{\kern1.4ex\char'47}}#1\relax\else\if l#1\relax
  \rlap{\raise.5\dimen7\hbox{\kern.4ex\char'47}}#1\relax \else\if
  L#1\relax\rlap{\raise.5\dimen7\hbox{\kern.7ex
      \char'47}}#1\relax\else\message{accent \string\soft \space #1
    not defined!}#1\relax\fi\fi\fi\fi\fi\fi}

\begin{document}

\title[Totally acyclic complexes and locally Gorenstein rings]{Totally
  acyclic complexes\\ and locally Gorenstein rings}

\author[L.\,W. Christensen]{Lars Winther Christensen}

\address{L.W.C. \ Texas Tech University, Lubbock, TX 79409, U.S.A.}

% \address{L.W.C.\newline\phantom{\llap{.}}\hspace{1em} Texas Tech
% University, Lubbock, TX 79409, U.S.A.}

\email{lars.w.christensen@ttu.edu}

\urladdr{http://www.math.ttu.edu/~lchriste}

\author[K. Kato]{Kiriko Kato}

\address{K.K. \ Osaka Prefecture University, Sakai, Osaka 599-8531,
  Japan}

\email{kiriko@mi.s.osakafu-y.ac.jp}

\urladdr{http://www.mi.s.osakafu-u.ac.jp/~kiriko}

\thanks{The work was done while L.W.C.\ visited the Department of
  Mathematics and Information Sciences at Osaka Prefecture University.
  The visit took place under the Visitor Program of the Graduate
  School of Science, and the support and hospitality is acknowledged
  with gratitude.}

\date{10 February 2017}



\begin{abstract}
  A commutative noetherian ring with a dualizing complex is Gorenstein
  if and only if every acyclic complex of injective modules is totally
  acyclic. We extend this characterization, which is due to Iyengar
  and Krause, to arbitrary commutative noetherian rings, i.e.\ we
  remove the assumption about a dualizing complex. In this context
  Gorenstein, of course, means locally Gorenstein at every prime.
\end{abstract}

\maketitle

\thispagestyle{empty}

\section{The Theorem}
\label{sec:introduction}

\noindent
Let $R$ be a commutative ring. A complex $A$ of $R$-modules with
$\H{A}=0$ is called \emph{acyclic}. An acyclic complex $P$ of
projective $R$-modules is called \emph{totally acyclic} if
$\Hom{P}{Q}$ is acyclic for every projective $R$-module $Q$; likewise
an acyclic complex $I$ of injective $R$-modules is called totally
acyclic if $\Hom{E}{I}$ is acyclic for every injective $R$-module
$E$. Finally, an acyclic complex $F$ of flat $R$-modules is called
\emph{F-totally acyclic} if $\tp{F}{E}$ is acyclic for every injective
$R$-module~$E$.

The invariant $\splf = \sup\{\projdim{F} \mid F\ \text{\rm is a flat
  $R$-module}\}$ is finite if and only if every flat $R$-module has
finite projective dimension. Indeed, a direct sum of flat modules is
flat with $\projdim(\bigoplus_{\lambda\in \Lambda}F_\lambda) =
\sup_{\lambda\in \Lambda}\{\projdim F_\lambda\}$. A ring with $\splf
\le n$ is called \emph{$n$-perfect}; in particular, a $0$-perfect ring
is a perfect ring in the sense of Bass' Theorem P. With these
definitions in place we can state our:

\begin{theorem}
  \label{thm:main}
  Let $R$ be a commutative noetherian ring. Among the conditions
  \begin{enumerate}
  \item The local ring $R_\fp$ is Gorenstein for every prime ideal
    $\fp$ of $R$
  \item Every acyclic complex of injective $R$-modules is totally
    acyclic
  \item Every acyclic complex of flat $R$-modules is F-totally acyclic
  \item Every acyclic complex of projective $R$-modules is totally
    acyclic
  \end{enumerate}
  the following implications hold
  \begin{equation*}
    (i) \iff (ii) \iff (iii) \implies (iv)\:.
  \end{equation*}
  Moreover, if $\splf$ is finite, then all four conditions are
  equivalent.
\end{theorem}
\noindent
The proof is given in the next section. Here we continue with a
discussion of the precursors of the theorem and the condition $\splf <
\infty$; in the rest of this section $R$ is assumed to be noetherian.

If the ring $R$ has finite Krull dimension---if it is local, in
particular---then it is Gorenstein if and only if $\injdim{R}$ is
finite. It is standard to call a ring that satisfies part $(i)$ in the
theorem Gorenstein, but notice that such a ring need not have finite
injective dimension as a module over itself. Nagata's regular ring of
infinite Krull dimension serves as an example.

Equivalence of the conditions $(i)$, $(ii)$, and $(iv)$ in the theorem
was proved by Iyengar and Krause \cite[Corollary 5.5]{SInHKr06} under
the assumption that $R$ has a dualizing complex. That assumption
implies that $R$ has finite Krull dimension; see Hartshorne
\cite[Corollary V.7.2]{rad}, and as we recall next that implies $\splf
< \infty$.

For a ring of finite Krull dimension $d$, one has $\splf \le d$ by
work of Gruson and Raynaud \cite[Theorem~II.(3.2.6)]{LGrMRn71} and
Jensen \cite[Proposition~6]{CUJ70}. In this context the equivalence of
conditions $(i)$, $(iii)$, and $(iv)$ in the theorem was proved by
Murfet and Salarian \cite[Theorem 4.27 \& Corollary 4.28]{DMfSSl11},
and the equivalence of all four conditions was recently proved by
Estrada, Fu, and Iacob \cite{EFI}.

Thus, the novelty of Theorem~\ref{thm:main} is twofold; it
establishes:
\begin{list}{$\bullet$}%
  {\setlength{\leftmargin}{2\parindent}%
    \setlength{\labelwidth}{\parindent}%
    \setlength{\rightmargin}{0pt}%
    \setlength{\partopsep}{0pt}%
    \setlength{\topsep}{0pt}%
    \setlength{\parsep}{0pt}%
    \setlength{\itemsep}{0pt}}%
\item the equivalence of conditions $(i)$--$(iii)$ for rings of
  infinite Krull dimension.

\item the equivalence of $(i)$--$(iv)$ for certain rings of infinite
  Krull dimension. Indeed, by a result of Gruson and
  Jensen~\cite[Theorem~7.10]{LGrCUJ81} one has $\splf \le n+1$ if $R$
  has cardinality at most $\aleph_n$ for some natural number $n$.
\end{list}

\section{The proof}
\renewcommand{\labelenumi}{(\arabic{enumi})}

\noindent
Our argument for $(ii) \Rightarrow (i)$ is an application of a result
due to \v{S}{\soft{t}}ov{\'{\i}}{\v{c}}ek \cite{JSt}. The rest of the
proof is based on, mostly, standard arguments that can be found in
some form in the literature. We repeat them for the reader's
convenience and to clarify under which conditions they apply.

\subsection*{Outline}

For any commutative ring $R$, one can consider the following
conditions:
\begin{enumerate}
\item Every acyclic complex of projective $R$-modules is F-totally
  acyclic.
\item Every acyclic complex of injective $R$-modules is totally
  acyclic.
\item Every acyclic complex of flat $R$-modules is F-totally acyclic.
\item Every acyclic complex of projective $R$-modules is totally
  acyclic.
\end{enumerate}
For a noetherian ring $R$ one can also consider:
\begin{enumerate}
\item[$(3^{\rm bis})$] The local ring $R_\fp$ is Gorenstein for every
  prime ideal $\fp$ of $R$.
\end{enumerate}

\hspace{-23.5pt} \renewcommand{\arraystretch}{1.25}
\begin{tabular}{rll}
  \multicolumn{3}{l}{The proof of Theorem \ref{thm:main} goes as follows. 
    These implications always hold:}\\
  $\qquad\qquad(1) \iff (3) \impliedby (2)$ 
  & see Claims \ref{claim1} and \ref{claim2}.\\
  \multicolumn{3}{l}{Under the assumption that $R$ is coherent, one has}\\
  $(1) \implies (4)$ 
  &see \cite[Lemma 4.20]{DMfSSl11} or Claim \ref{claim3},\\
  \multicolumn{3}{l}{and if $R$ is coherent with $\splf$ finite, 
    then one also has}\\
  $(1) \impliedby (4)$ 
  &see \cite[Lemma 4.20]{DMfSSl11} or Claim \ref{claim4}.\\
  \multicolumn{3}{l}{If $R$ is noetherian, then one has}\\
  $(3^{\rm bis}) \iff (3)$ 
  &see \cite[Theorem 4.27]{DMfSSl11} or Claim \ref{claim5},\\
  \multicolumn{3}{l}{and}\\
  $(3^{\rm bis}) \implies (2)$ 
  &see Claim \ref{claim6}.\\
  \multicolumn{3}{l}{This finishes the proof of Theorem~\ref{thm:main}.}
\end{tabular}

\begin{center}
  $\ast \ \ast \ \ast$
\end{center}

\noindent
Recall that an $R$-module $C$ is called \emph{cotorsion} if
$\Ext{1}{F}{C} = 0$ holds for every flat $R$-module $F$.

\begin{claim}
  \label{claim6}
  Let $R$ be noetherian. If the local ring $R_\fp$ is Gorenstein for
  every prime ideal $\fp$ in $R$, then every acyclic complex of
  injective $R$-modules is totally acyclic.
\end{claim}

\begin{proof}
  Let $I$ be an acyclic complex of injective $R$-modules. It is
  sufficient to show that $\Hom{E}{I}$ is acyclic for every
  indecomposable injective $R$-module $E = E(R/\fp)$. Fix such a
  module, it has an $R_\fp$-module structure, and the $R$-action
  factors through this structure, so standard adjunction yields
  \begin{equation*}
    \Hom{E}{I} \is \Hom{\tp[R_\fp]{R_\fp}{E}}{I} \is 
    \Hom[R_\fp]{E}{\Hom{R_\fp}{I}}\:.
  \end{equation*}
  By \cite[Corollary~5.9]{JSt} the cycle submodules of the complex $I$
  are cotorsion, so the complex $\Hom{R_\fp}{I}$ of injective
  $R_\fp$-modules is acyclic. As $R_\fp$ is Gorenstein, the injective
  module $E$ has finite projective dimension, and it follows that
  $\Hom[R_\fp]{E}{\Hom{R_\fp}{I}}$ and hence $\Hom{E}{I}$ acyclic.
\end{proof}

\begin{claim}
  \label{claim1}
  Every acyclic complex of projective $R$-modules is F-totally acyclic
  if and only if every acyclic complex of flat $R$-modules is
  F-totally acyclic.
\end{claim}

\begin{proof}
  The ``if'' part is trivial as projective modules are flat. For the
  converse, let $F$ be an acyclic complex of flat modules. It follows
  from \cite[Propositions 8.1 \& 9.1]{ANm08} that, in the homotopy
  category of complexes of flat modules, there is a distinguished
  triangle $P \to F \to C \to $, where $P$ is a complex of projective
  modules and $C$ is pure acyclic; i.e.\ $\H{C} = 0 = \H{\tp{C}{M}}$
  for every module $M$.  It follows that $P$ is acyclic, so by
  assumption it is F-totally acyclic, and hence so is $F$.
\end{proof}

\begin{claim}
  \label{claim2}
  If every acyclic complex of injective $R$-modules is totally
  acyclic, then every acyclic complex of flat $R$-modules is F-totally
  acyclic.
\end{claim}

\begin{proof}
  Let $F$ be an acyclic complex of flat $R$-modules. The character
  complex $\Hom[\ZZ]{F}{\QZ}$ is an acyclic complex of injective
  $R$-modules, so the complex
  \begin{equation*}
    \Hom{E}{\Hom[\ZZ]{F}{\QZ}} \is \Hom[\ZZ]{\tp{F}{E}}{\QZ}
  \end{equation*}
  is by assumption acyclic for every injective $R$-module $E$. It
  follows that $\tp{F}{E}$ is acyclic for every injective $R$-module
  $E$.
\end{proof}

The classes of flat modules and cotorsion modules make up a complete
cotorsion pair; that is a theorem of Bican, El Bashir, and Enochs
\cite{BEE-01}. In particular, there is for every $R$-module $M$ an
exact sequence $0 \to M \to C \to F \to 0$ where $C$ is cotorsion and
$F$ is flat. If $M$ is flat, then so is $C$, so by iteration a flat
$R$-module admits a, possibly infinite, coresolution by flat cotorsion
modules.

\begin{claim}
  \label{claim3}
  Let $R$ be coherent and $P$ be an acyclic complex of projective
  $R$-modules. If $P$ is F-totally acyclic, then $\Hom{P}{G}$ is
  acyclic for every flat $R$-module $G$; in particular, $P$ is totally
  acyclic.
\end{claim}

\begin{proof}
  For every injective $R$-module $E$ the complex $\tp{P}{E}$ and hence
  \begin{equation*}
    \Hom[\ZZ]{\tp{P}{E}}{\QZ} \is \Hom{P}{\Hom[\ZZ]{E}{\QZ}}
  \end{equation*}
  is acyclic. By \cite[Lemma 3.2.3]{xu} an $R$-module is cotorsion and
  flat if and only if it is a direct summand of the character module
  $\Hom[\ZZ]{E}{\QZ}$ of some injective $R$-module $E$. Thus,
  $\Hom{P}{C}$ is acyclic for every flat cotorsion $R$-module $C$.

  Let $G$ be a flat $R$-module with cotorsion coresolution $G \to
  C$. The fact that $\Hom{P}{C^i}$ is acyclic for every $i$ implies by
  \cite[Lemma~2.5]{CFH-06} that $\Hom{P}{C}$ is acyclic. Consider the
  augmented coresolution
  \begin{equation*}
    G^+ \,=\; 0  \lra G \lra C^0 \lra C^1 \lra \cdots 
    \lra C^i \lra C^{i+1} \lra \cdots\:.
  \end{equation*}
  It is a pure acyclic complex of flat $R$-modules, so by
  \cite[Theorem~8.5]{ANm08} the complex $\Hom{P}{G^+}$ is
  acyclic. Since the complex $G^+$ is the mapping cone of the morphism
  $G \to C$, one now has $\H{\Hom{P}{G}} \is \H{\Hom{P}{C}} =0$.
\end{proof}

\begin{claim}
  \label{claim4}
  Let $R$ be coherent with $\splf$ finite. Every totally acyclic
  complex $P$ of projective $R$-modules is F-totally acyclic.
\end{claim}

\begin{proof}
  Let $P$ be a totally acyclic complex of projective $R$-modules and
  $E$ be an injective $R$-module. The complex $\tp{P}{E}$ is acyclic
  if the character complex
  \begin{equation*}
    \Hom[\ZZ]{\tp{P}{E}}{\QZ} \is \Hom{P}{\Hom[\ZZ]{E}{\QZ}}
  \end{equation*}
  is acyclic. The $R$-module $\Hom[\ZZ]{E}{\QZ}$ is flat, see
  \cite[Lemma 3.1.4]{xu}, so by assumption it has finite projective
  dimension. As $\Hom{P}{Q}$ is acyclic for every projective
  $R$-module $Q$, it follows that $\Hom{P}{\Hom[\ZZ]{E}{\QZ}}$ is
  acyclic.
\end{proof}

\begin{claim}
  \label{claim5}
  Let $R$ be noetherian. Every acyclic complex of flat $R$-modules is
  totally acyclic if and only if $R_\fp$ is Gorenstein for every prime
  ideal $\fp$ of $R$.
\end{claim}
% \enlargethispage*{2\baselineskip}

\begin{proof}
  ``If'': Let $F$ be an acyclic complex of flat $R$-modules and $E$ be
  an injective $R$-module. The complex $\tp{F}{E}$ is acyclic if
  $(\tp{F}{E})_\fp \is \tp[R_\fp]{F_\fp}{E_\fp}$ is acyclic for every
  prime $\fp$. Fix $\fp$; the complex $F_\fp$ of flat $R_\fp$-modules
  is acyclic and, since $R_\fp$ is Gorenstein, the injective
  $R_\fp$-module $E_\fp$ has finite flat dimension; hence
  $\tp[R_\fp]{F_\fp}{E_\fp}$ is acyclic.

  ``Only if'': Fix a prime $\fp$. The local ring $R_\fp$ is Gorenstein
  if its $\fp R_\fp$-adic completion $\widehat{R_\fp}$ is Gorenstein.
  As an $R$-module $\widehat{R_\fp}$ is flat, so every injective
  $\widehat{R_\fp}$-module is injective over $R$.  Let $P$ be an
  acyclic complex of projective $\widehat{R_\fp}$-modules; it is a
  complex of flat $R$-modules and hence F-totally acyclic. Thus, for
  every injective $\widehat{R_\fp}$-module $E$ the complex $\tp{P}{E}
  \is \tp[\widehat{R_\fp}]{P}{E}$ is acyclic. That is, $P$ is
  F-totally acyclic and hence totally acyclic by
  Claim~\ref{claim3}. Being a complete local ring, $\widehat{R_\fp}$
  has a dualizing complex, and it follows from \cite[Corollary
  5.5]{SInHKr06} that it is Gorenstein.
\end{proof}

\begin{remark}
  \label{remark}
  We do not know if condition (\ref{p}) below is equivalent to
  finiteness of $\splf$; it is not even clear that it implies finite
  projective dimension of the flat $R$-module $\prod_\fp R_\fp$.
  \begin{equation}
    \label{p}
    \projdim{R_\fp} < \infty \ \ \text{for every prime ideal $\fp$ of $R$}\,.
  \end{equation}
  We show in Claim~\ref{claim7} that $(4) \Rightarrow (1)$ holds if
  $R$ is noetherian and satisfies (\ref{p}). Thus all four conditions
  in Theorem~\ref{thm:main} are equivalent for such rings.
\end{remark}

\begin{claim}
  \label{claim7}
  Let $R$ be noetherian. If $\projdim{R_\fp}$ is finite for every
  prime ideal $\fp$ of $R$, then every totally acyclic complex of
  projective $R$-modules is F-totally acyclic.
\end{claim}

\begin{proof}
  Let $P$ be a totally acyclic complex of projective $R$-modules and
  let $E$ be an injective $R$-module. The complex $\tp{P}{E}$ is
  acyclic if $(\tp{P}{E})_\fp \is \tp{P}{E_\fp}$ is acyclic for every
  prime $\fp$. Fix $\fp$; it is sufficient to prove that the character
  complex
  \begin{equation*}
    \Hom[\ZZ]{\tp{P}{E_\fp}}{\QZ} \is \Hom{P}{\Hom[\ZZ]{E_\fp}{\QZ}}
  \end{equation*}
  is acyclic. The $R_\fp$-module $\Hom[\ZZ]{E_\fp}{\QZ}$ is flat, so
  it has finite projective dimension over $R_\fp$. Every projective
  $R_\fp$-module has finite projective dimension over $R$ by the
  assumption $\projdim{R_\fp} < \infty$. It follows that
  $\projdim{\Hom[\ZZ]{E_\fp}{\QZ}}$ is finite. By assumption
  $\Hom{P}{Q}$ is acyclic for every projective $R$-module $Q$, and it
  follows that $\Hom{P}{\Hom[\ZZ]{E}{\QZ}}$ is acyclic.
\end{proof}

\bibliographystyle{amsplain}

% \bibliography{../+references}

\begin{thebibliography}{10}

\bibitem{BEE-01} Ladislav Bican, Robert El~Bashir, and
  Edgar~E. Enochs, \emph{All modules have flat covers}, Bull. London
  Math. Soc. \textbf{33} (2001), no.~4, 385--390.  \MR{MR1832549}

\bibitem{CFH-06} Lars~Winther Christensen, Anders Frankild, and Henrik
  Holm, \emph{On {G}orenstein projective, injective and flat
    dimensions---{A} functorial description with applications},
  J.~Algebra \textbf{302} (2006), no.~1, 231--279. \MR{MR2236602}

\bibitem{EFI} Sergio Estrada, Xianhui Fu, and Alina Iacob,
  \emph{Totally acyclic complexes}, Preprint \arxiv[AC]{1603.03850v2}.

\bibitem{LGrCUJ81} Laurent Gruson and Christian~U. Jensen,
  \emph{Dimensions cohomologiques reli\'ees aux foncteurs
    {$\varprojlim\sp{(i)}$}}, Paul {D}ubreil and {M}arie-{P}aule
  {M}alliavin {A}lgebra {S}eminar, 33rd {Y}ear ({P}aris, 1980),
  Lecture Notes in Math., vol. 867, Springer, Berlin, 1981,
  pp.~234--294.  \MR{MR0633523}

\bibitem{rad} Robin Hartshorne, \emph{Residues and duality}, Lecture
  notes of a seminar on the work of A. Grothendieck, given at Harvard
  1963/64. With an appendix by P.  Deligne. Lecture Notes in
  Mathematics, vol.~20, Springer-Verlag, Berlin, 1966. \MR{MR0222093}

\bibitem{SInHKr06} Srikanth Iyengar and Henning Krause,
  \emph{Acyclicity versus total acyclicity for complexes over
    {N}oetherian rings}, Doc. Math. \textbf{11} (2006),
  207--240. \MR{MR2262932}

\bibitem{CUJ70} Christian~U. Jensen, \emph{On the vanishing of
    {$\underset{\longleftarrow}{\lim}^{(i)}$}}, J.~Algebra \textbf{15}
  (1970), 151--166. \MR{MR0260839}

\bibitem{DMfSSl11} Daniel Murfet and Shokrollah Salarian,
  \emph{Totally acyclic complexes over {N}oetherian schemes},
  Adv. Math. \textbf{226} (2011), no.~2, 1096--1133.  \MR{MR2737778}

\bibitem{ANm08} Amnon Neeman, \emph{The homotopy category of flat
    modules, and {G}rothendieck duality}, Invent. Math. \textbf{174}
  (2008), no.~2, 255--308. \MR{MR2439608}

\bibitem{LGrMRn71} Michel Raynaud and Laurent Gruson, \emph{Crit\`eres
    de platitude et de projectivit\'e. {T}echniques de
    ``platification'' d'un module}, Invent. Math.  \textbf{13} (1971),
  1--89. \MR{MR0308104}

\bibitem{JSt} Jan {\v{S}}{\soft{t}}ov{\'{\i}}{\v{c}}ek, \emph{On
    purity and applications to coderived and singularity categories},
  preprint, \arxiv[CT]{1412.1615v1}.

\bibitem{xu} Jinzhong Xu, \emph{Flat covers of modules}, Lecture Notes
  in Mathematics, vol.  1634, Springer-Verlag, Berlin,
  1996. \MR{MR1438789}

\end{thebibliography}

\providecommand{\MR}[1]{\mbox{\href{http://www.ams.org/mathscinet-getitem?mr=#1}{#1}}}
\renewcommand{\MR}[1]{\mbox{\href{http://www.ams.org/mathscinet-getitem?mr=#1}{#1}}}
\providecommand{\arxiv}[2][AC]{\mbox{\href{http://arxiv.org/abs/#2}{\sf
      arXiv:#2 [math.#1]}}} \def\cprime{$'$}
\providecommand{\bysame}{\leavevmode\hbox to3em{\hrulefill}\thinspace}
\providecommand{\MR}{\relax\ifhmode\unskip\space\fi MR }
\providecommand{\MRhref}[2]{%
  \href{http://www.ams.org/mathscinet-getitem?mr=#1}{#2} }
\providecommand{\href}[2]{#2}

\end{document}